\documentclass[12pt, notitlepage]{amsart}
\usepackage{latexsym, amsfonts, amsmath, amssymb, amsthm, cite}

\newtheorem{theorem}{Theorem}
\newtheorem{lemma}[theorem]{Lemma}
\newtheorem{proposition}[theorem]{Proposition}

\usepackage{graphicx}
\usepackage{mathrsfs}

\begin{document}

\title{Equidistribution of zeros of polynomials} 
 \author{K. Soundararajan} 
\address{Department of Mathematics \\ Stanford University \\
450 Serra Mall, Bldg. 380\\ Stanford, CA 94305-2125}
\email{ksound@math.stanford.edu}
\thanks{} 
\begin{abstract}{A classical result of Erd{\H o}s and Tur{\' a}n states that if a monic polynomial has 
small size on the unit circle and its constant coefficient is not too small, then its zeros cluster near the unit circle and become equidistributed 
in angle.   Using Fourier analysis we give a short and self-contained proof of this result.}   \end{abstract} 
  
 \maketitle

 \section{Introduction.}  
 
 Any set of $N$ complex numbers may be viewed as the zero set of a polynomial of degree $N$.  If, however, 
 we start with a polynomial that ``arises naturally"\textemdash for example, think of polynomials with coefficients $\pm 1$\textemdash then the 
 zeros will tend to be ``evenly distributed near the unit circle."  
 In \cite{ErdosTuran}, Erd{\H o}s and Tur{\' a}n proved the beautiful result that if the size of a monic polynomial on the 
 unit circle is small, and its constant term is not too small, then its zeros cluster around the unit circle and become equally distributed in sectors.  We shall make precise 
 both the hypothesis and conclusion of this statement later, but we hope Figure 1 gives an impression of the phenomenon.  
 The Erd{\H o}s-Tur{\' a}n result was subsequently refined by Ganelius \cite{Ganelius} and Mignotte \cite{Mignotte}, 
and in this note we give a short and self-contained proof, obtaining as a bonus a modest improvement of the previous results.

 %

Let 
$$ 
P(z) = \prod_{j=1}^{N} (z-\alpha_j) = z^N + a_{N-1} z^{N-1} + \cdots + a_0 
$$ 
be a polynomial of degree $N$, and write the roots $\alpha_j$ as $\alpha_j =\rho_j e^{i\theta_j}$.   It may be helpful 
to think first of situations where the roots are {\sl not} equidistributed near the unit circle.  For example, one could have 
the polynomial $(z-1)^N = \sum_{j=0}^{N} (-1)^j \binom{N}{j} z^j$, where all the roots are concentrated at one point $z=1$ 
and clearly not spread out evenly.  This polynomial has large coefficients, and on the unit circle it attains a maximum size of $2^N$.    
A different type of example is the polynomial $z^N -1/2^N$.  Here the polynomial takes only small values on the unit circle, 
but all the roots are on the circle with radius $1/2$.   A more extreme version of this example is the polynomial $z^N$.

\begin{figure} \label{figure1}
\includegraphics[scale=.35]{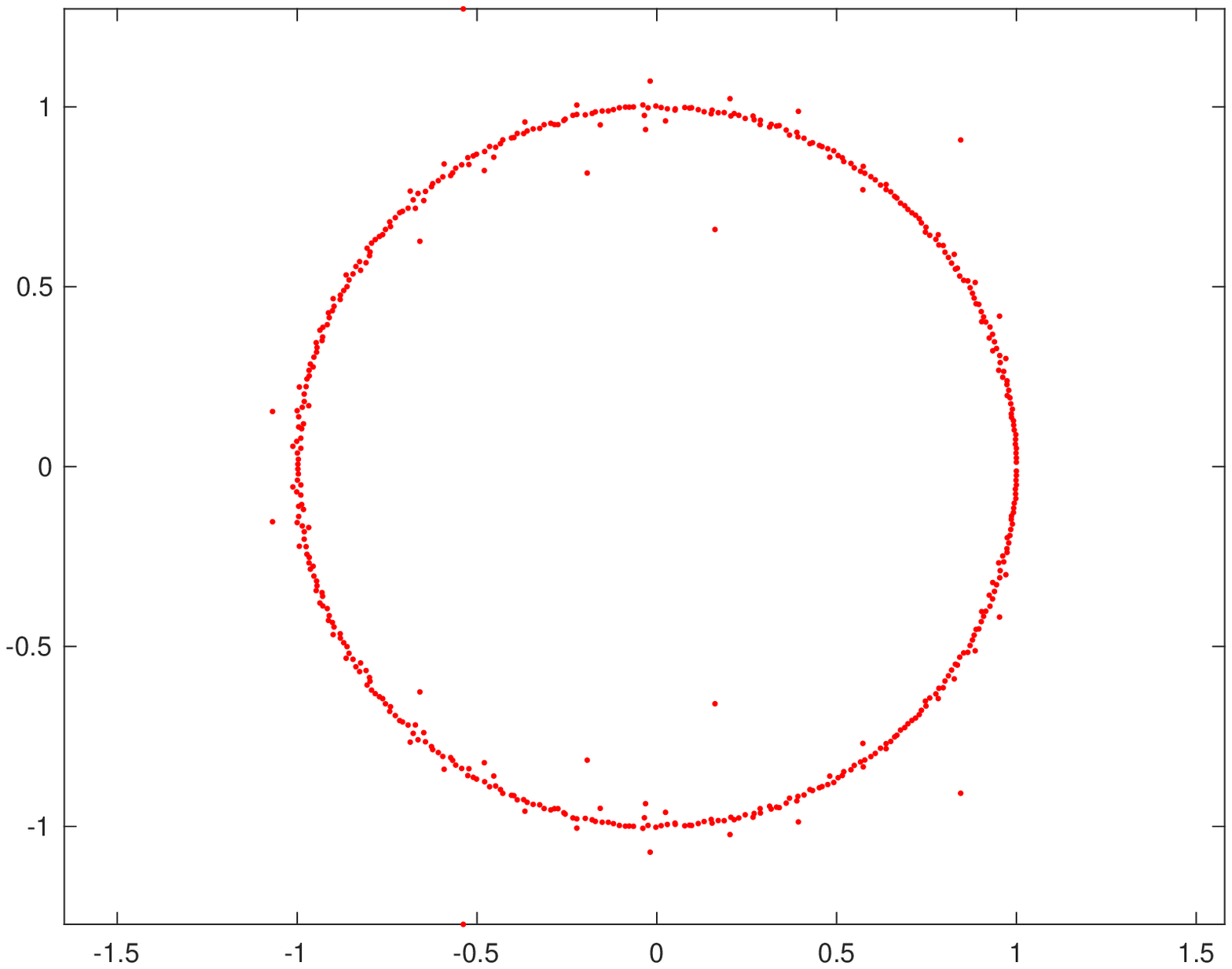} \qquad 
\includegraphics[scale=.35]{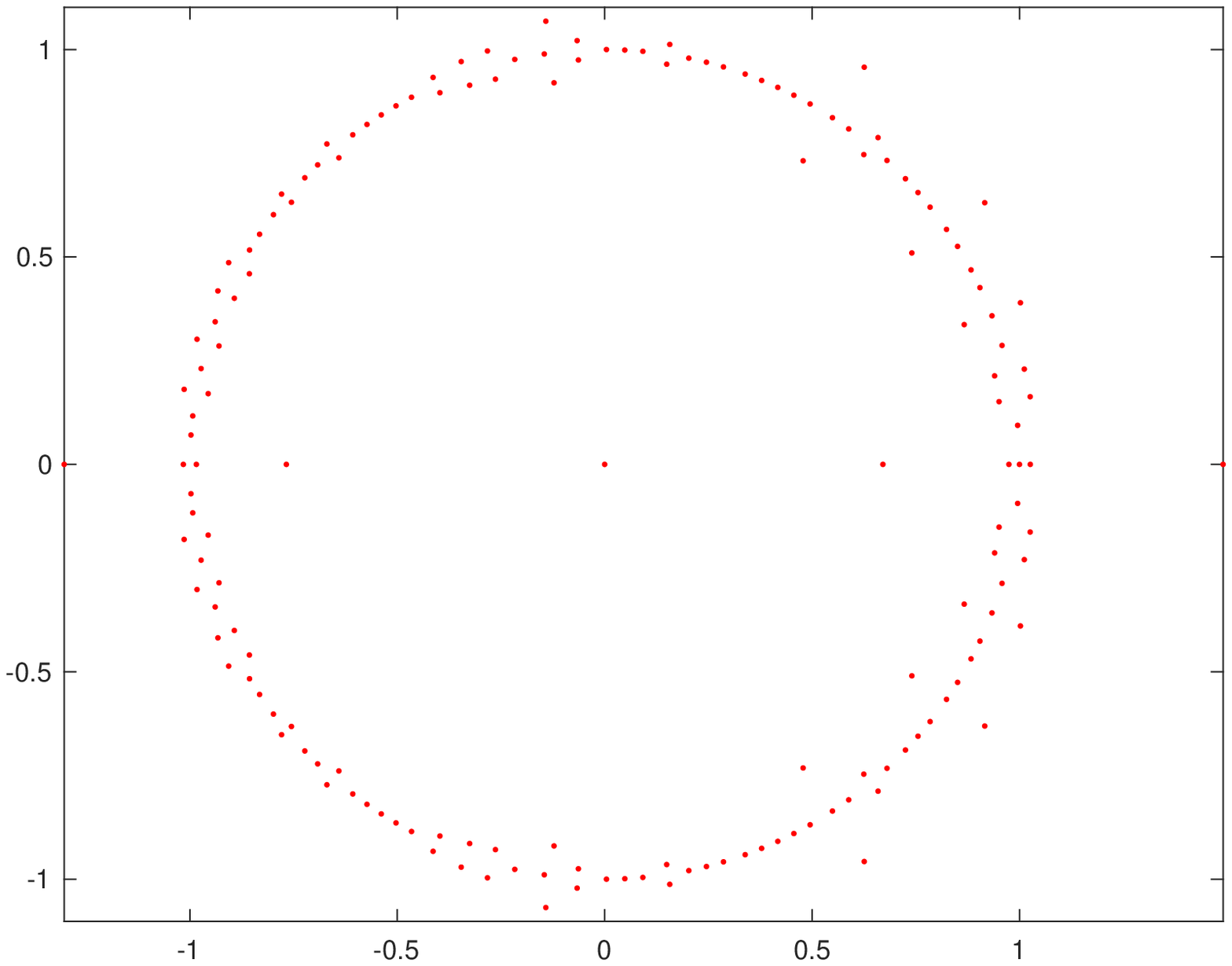} 
\caption{Left:  Zeros of a polynomial of degree $500$ formed with the decimal digits of $\pi$:  $3z^{500}+z^{499}+4z^{498}+\cdots$.  
Right:  Zeros of the Fekete polynomial $\sum_{j=0}^{162} (\frac{j}{163}) z^j$ where the coefficients are given by the Legendre symbol 
$(\frac{j}{163}) = \pm 1$ for $1\le j \le 162$. }
  \end{figure}

These examples indicate that it would be necessary to assume that the size of $P$ on the unit circle must be small, and that 
the constant term $a_0$ should not be too small in order to establish equidistribution of zeros.  Henceforth we 
 will assume that $a_0\neq 0$ so that the roots $\alpha_j  $ are all nonzero. 
 One convenient 
 measure of the size of coefficients is the quantity 
 \begin{equation*} 
 \label{1} 
 H(P) = \max_{|z| = 1} \frac{|P(z)|}{\sqrt{|a_0|}}. 
 \end{equation*}  
The triangle inequality gives, with the convention $a_N=1$, 
$$ 
H(P) \le \frac{1}{\sqrt{|a_0|}} \sum_{j=0}^{N} |a_j|. 
$$ 
On the other hand, Parseval's formula gives 
$$ 
H(P)^2 = \max_{|z|=1} \frac{|P(z)|^2}{|a_0|} \ge \frac{1}{|a_0|} \frac{1}{2\pi }\int_0^{2\pi} |P(e^{i\theta})|^2 d\theta = 
\frac{1}{|a_0|} \sum_{j=0}^{N} |a_j|^2. 
$$ 
Combining our upper bound for $H(P)$ with the Cauchy--Schwarz inequality we find that 
\begin{equation} 
\label{2} 
\frac{1}{|a_0|} \sum_{j=0}^{N} |a_j|^2 \le  H(P)^2 \le \frac{N}{|a_0|} \sum_{j=0}^{N} |a_j|^2. 
\end{equation} 
Assuming that $H(P)$ is small is therefore equivalent to assuming that the coefficients of $P$ are small and that 
the constant coefficient $a_0$ is not too small.  Here by ``small" we mean that $H(P)$ is not exponentially large in $N$; 
for example, one could think of a condition like $H(P) \le e^{\epsilon N}$ for suitably small $\epsilon$.  In fact, we shall formulate 
the Erd{\H o}s--Tur{\' a}n theorem in terms of the slightly more refined quantity 
\begin{equation*} 
\label{1.1}  
h(P) = \frac{1}{2\pi } \int_0^{2\pi } \log^+ \frac{|P(e^{i\theta})|}{\sqrt{|a_0|}} d\theta, 
\qquad\text{where} \qquad \log^+ x= \max(0, \log x). 
\end{equation*} 
Since $H(P) \ge (|a_0|^2 + |a_N|^2)/|a_0| = (|a_0|^2 +1)/|a_0| \ge 1$ in view of the lower bound in \eqref{2}, the quantity $h(P)$ satisfies 
$$ 
h(P) \le \log \max_{|z| =1} \frac{|P(z)|}{\sqrt{|a_0|}} = \log H(P), 
$$ 
so that the assumption that $h(P)$ is small is weaker than the assumption that $H(P)$ is small. 

We now turn to the question of how to quantify the idea that zeros are equidistributed around the unit circle.  
We do this in two stages, first discussing the magnitude of zeros, and then discussing the spacings of their 
arguments.  To treat the magnitude of the zeros (recall $\alpha_j=\rho_j e^{i\theta_j}$), we define  
\begin{equation*} 
\label{4} 
{\mathcal M}(P) = \prod_{j=1}^{N} \max\Big( \rho_j, \frac{1}{\rho_j}\Big). 
\end{equation*} 
As an easy consequence of Jensen's formula from complex analysis, we shall establish the 
following upper bound for ${\mathcal M}(P)$ in terms of $h(P)$.  

\begin{theorem} \label{thm1}  With notations as above,
$$ 
{\mathcal M}(P) \le \exp(2 h(P)). 
$$ 
\end{theorem} 

To gain a sense of this result, suppose we knew the upper bound ${\mathcal M}(P) \le \exp(\epsilon^2 N)$.  Then 
it would follow that at most $\epsilon N$ zeros can lie outside the band $e^{-\epsilon} \le |z| \le e^{\epsilon}$.   Or, in 
other words, most of the zeros will lie inside a narrow band around the unit circle.

The more difficult part of the Erd{\H o}s--Tur{\' a}n theorem concerns the equidistribution of the angles $\theta_j$.     
 Given an arc $I$ on the unit circle, let $N(I;P)$ denote the number of zeros $\alpha_j$ with $e^{i\theta_j}$ lying on this arc.   If the angles $\theta_j$ are equidistributed, then we may expect $N(I;P)$ to be roughly $\frac{N}{2\pi}$ times the length of the arc $I$ (which we will denote by $|I|$).    A convenient way to measure equidistribution is the 
 {\sl discrepancy}, which is defined as 
 \begin{equation*} 
 \label{1.2} 
{\mathcal D}(P) =  \max_{I} \Big|  N(I;P) - \frac{|I|}{2\pi} N\Big| . 
\end{equation*}  
In other words, the discrepancy measures the {\sl worst case} deviation of the actual count of the 
number of angles lying on a given arc from the number that one would expect if the angles were equidistributed.  A bound ${\mathcal D}(P) \le \epsilon N$, for suitably small $\epsilon$, would indicate that the angles $\theta_j$ are evenly distributed.

\begin{theorem} \label{thm2}  With notations as above, 
\begin{equation} 
\label{1.4}
{\mathcal D}(P) \le \frac{8}{\pi} \sqrt{N h(P)}. 
\end{equation}
\end{theorem}

 
 Theorems \ref{thm1} and \ref{thm2} together establish that if $h(P)$ is small compared to $N$, then 
 the zeros of $P$ cluster around the unit circle and become equidistributed in angle.  For example, 
 if the coefficients of $P$ are always $\pm 1$, then since $H(P) \le N+1$, it follows from Theorem \ref{thm1} 
 that ${\mathcal M}(P) \le (N+1)^2$, and from Theorem \ref{thm2} that ${\mathcal D}(P) \le \frac{8}{\pi} \sqrt{N \log (N+1)}$. 
 
 %
  Erd{\H o}s and Tur{\' a}n \cite{ErdosTuran} first established a version of \eqref{1.4}, with the constant $8/\pi$ replaced by $16$ and with $\log H(P)$ instead of $h(P)$.  Ganelius \cite{Ganelius} showed an estimate like \eqref{1.4}, again with $\log H(P)$ instead of $h(P)$, but with a better constant than Erd{\H o}s and Tur{\' a}n, namely with $\sqrt{2\pi/k} = 2.5619\ldots$ (with $k=1/1^2-1/3^2+1/5^2-\cdots=0.9159\ldots $ denoting the Catalan constant) instead of $16$.  Mignotte \cite{Mignotte}  refined Ganelius's result, replacing $\log H(P)$ by the sharper $h(P)$.  Note that our theorem sharpens the Ganelius--Mignotte result slightly, since $8/\pi = 2.5464\ldots$ is a little smaller than $\sqrt{2\pi/k}$.   There is some scope to improve the constant $8/\pi$ 
  (especially in the situation  where $h(P)$ is small compared to $N$), but Amoroso and Mignotte \cite{AM} have produced examples showing that the constant in \eqref{1.4} must be at least $\sqrt{2}$.

There is a vast literature surrounding zeros of polynomials, and we give a few references to related work.    For the 
distribution of zeros of polynomials with $0$, $1$ coefficients see \cite{OP};   for work on ``Fekete polynomials" where 
the coefficients equal the Legendre symbol $\pmod p$, see \cite{CGPS}; for work on random polynomials with coefficients 
drawn independently from various distributions (where Theorems \ref{thm1} and \ref{thm2} will apply with high probability), see \cite{HN}; for two recent variants on the Erd{\H o}s--Tur{\' a}n theorem, see \cite{TV} and \cite{Erdelyi}.    While the Erd{\H o}s--Tur{\' a}n result applies to all polynomials 
with complex coefficients, in number theory greater interest is attached to irreducible polynomials with integer coefficients.  If 
$P(x) = a_N x^N+ \cdots +a_0 \in {\Bbb Z}[x]$ is a polynomial with roots $\alpha_j$, then a central object here is the {\sl Mahler measure} 
which is $M(P) = |a_N| \prod_{j=1}^{N} \max(1, |\alpha_j|)$.  A beautiful result of Bilu \cite{Bilu} states that if $P$ is an irreducible 
polynomial in ${\Bbb Z}[x]$ and $M(P) \le (1+\epsilon)^N$ is not large, then the zeros of $P$ cluster near the unit circle and 
are equidistributed; for a gentle exposition, see \cite{Granville}.   Any discussion of zeros of polynomials would be incomplete without a mention of Lehmer's outstanding open problem that 
the smallest value of $M(P)$ that is larger than $1$ is $M(L) =1.1762\ldots$, attained for Lehmer's polynomial $L(x) = x^{10}+x^{9} -x^7 -x^6-x^5-x^4-x^3+x+1$;  for a recent comprehensive survey, see \cite{Smyth}.   
Finally, our proof uses ideas from Fourier analysis; two lovely 
books in this area are \cite{Korner} and \cite{Montgomery}.

  \section{Jensen's formula and the proof of Theorem \ref{thm1}.}  
  
We begin with the easier result, Theorem \ref{thm1}, which follows from Jensen's formula.  If $f$ is holomorphic in a domain containing the 
unit disk with $f(0) \neq 0$, then Jensen's formula (see 5.3.1 of Ahlfors \cite{Ahlfors}) states that 
$$ 
\frac{1}{2\pi } \int_0^{2\pi} \log |f(e^{i\theta})| d\theta = \log |f(0)| + \sum_{j} \log \frac{1}{|z_j|}, 
$$
where the sum is over the zeros $z_j$ of $f$ lying inside the unit disk.  
  
 Applying Jensen's formula, and since $P(0)= a_0$, we find 
  $$ 
  \frac{1}{2\pi } \int_0^{2\pi} \log \frac{|P(e^{i\theta})|}{\sqrt{|a_0|} } d\theta = \frac 12\log |a_0| + \sum_{\rho_j <1} \log \frac{1}{\rho_j}.
  $$ 
  But $|a_0| = \prod_{j=1}^{N} \rho_j$, and so the above also equals 
  $$ 
  -\frac 12 \log |a_0| + \sum_{\rho_j >1} \log \rho_j. 
  $$ 
  Adding these two expressions, 
  \begin{equation*} 
  \label{1.31}
    2 \Big(\frac{1}{2\pi } \int_0^{2\pi} \log \frac{|P(e^{i\theta})|}{\sqrt{|a_0|}} d\theta \Big)= \sum_{j} \log \max \Big(\frac{1}{\rho_j}, \rho_j\Big)
  \end{equation*} 
  and, since the left side above is clearly at most $2 h(P)$, Theorem \ref{thm1} follows.  
   
   \section{An observation of Schur.}

  The rest of this article is devoted to proving Theorem \ref{thm2}.   We begin with an observation attributed to Schur (see \cite{AM}) that will allow us to restrict attention to polynomials with all zeros on the unit circle.  
 
  \begin{lemma} Let $P(z) =\prod_{j=1}^{N} (z-\alpha_j)$ with $\alpha_j=\rho_j e^{i\theta_j}$ be as above, and define the polynomial $Q$ by $Q(z) = \prod_{j=1}^{N} (z- e^{i\theta_j})$.   Then for 
  any $z$ with $|z|=1$, we have 
  $$ 
  \frac{|P(z)|}{\sqrt{|a_0|}} \ge |Q(z)|, 
  $$ 
  so that $h(P) \ge h(Q)$. 
  \end{lemma} 
  \begin{proof}  Observe that for any $z$ with $|z|=1$,  
    $$
  \Big| \frac{z}{\sqrt{\rho_j}} - \sqrt{\rho_j} e^{i\theta_j}\Big|^2 = \frac{1}{\rho_j}+\rho_j -2\text{Re } z e^{-i\theta_j} 
  \ge 2 -2 \text{Re }z e^{-i\theta_j} = |z-e^{i\theta_j}|^2, 
  $$ 
and so
  $$ 
  \frac{|P(z)|}{\sqrt{|a_0|}} = \prod_{j=1}^{N} \Big| \frac{z}{\sqrt{\rho_j} }-\sqrt{\rho_j}e^{i\theta_j} \Big| \ge \prod_{j=1}^{N} |z-e^{i\theta_j}| = |Q(z)|,
  $$  
  proving the lemma.
  \end{proof} 
  
 Since the discrepancies ${\mathcal D}(P)$ and ${\mathcal D}(Q)$ are the same, and since $h(P) \ge h(Q)$,  
  it is enough to establish Theorem \ref{thm2} for the polynomial $Q$ and then the corresponding bound for the 
  polynomial $P$ would follow.   In other words, we may assume from now on that all zeros of $P$ lie on the unit circle, so that $\rho_j=1$ for all $j$.

  \section{Smoothed sums over the zeros.} 
  
  Let $P(z) = \prod_{j=1}^{N} (z-e^{i\theta_j})$ be a polynomial of degree $N$ with all zeros on the 
  unit circle. The following lemma establishes a crucial link between the power sums of the 
  zeros (by which we mean $\sum_{j=1}^{N} e^{ik\theta_j}$ for integers $k$) and the size of $P$ on the 
  unit circle.  
  
   \begin{lemma}  \label{lem2}  Let $P(z) = \prod_{j=1}^{N} (z-e^{i\theta_j})$   be as above.  For any integer 
  $k\neq 0$ we have 
  \begin{equation} 
  \label{4.1}  
  \sum_{j=1}^{N} e^{ik\theta_j} =  -\frac{|k|}{\pi}  \int_0^{2\pi} e^{ik\theta} \log |P(e^{i\theta})|  d\theta . 
  \end{equation} 
  Consequently, for any integer $k\neq 0$, 
  \begin{equation} 
  \label{4.2} 
  \Big| \sum_{j=1}^{N} e^{ik\theta_j} \Big| \le 4|k| h(P). 
  \end{equation} 
  \end{lemma} 
  
  The link between power sums of the roots and the size of $P$ should not come as a surprise---Newton's 
  identities connecting power sums of the roots with the coefficients of a polynomial are a different version 
  of such a link.  For our purposes the identity \eqref{4.1} will, however, be much more useful than Newton's identities.  
  For a general polynomial $P(z) =\prod_{j=1}^{N} (z-\rho_j e^{i\theta_j})$, the relation \eqref{4.1} may be replaced 
  by 
  $$ 
  \sum_{j=1}^{N} \min \Big(\rho_j , \frac{1}{\rho_j}\Big)^{|k|} e^{ik\theta_j} =  -\frac{|k|}{\pi}  \int_0^{2\pi} e^{ik\theta} \log \frac{|P(e^{i\theta})|}{\sqrt{|a_0|}}  d\theta .
  $$

  \begin{proof}[Proof of Lemma \ref{lem2}]  For  any real number $\phi$ and nonzero integer $k$, we shall show that 
  \begin{equation} 
  \label{lem2.1} 
  e^{ik\phi} = -\frac{|k|}{\pi } \int_0^{2\pi } e^{ik\theta} \log |e^{i\theta} - e^{i\phi}| d\theta, 
  \end{equation} 
  and then \eqref{4.1} follows upon summing this over all $\phi = \theta_j$.  Substituting $\theta =x+\phi$, and dividing both sides by 
  $e^{ik\phi}$,  we see that \eqref{lem2.1} follows from the identity 
  \begin{align} 
  \label{lem2.2} 
  1&= -\frac{|k|}{\pi} \int_0^{2\pi} e^{ikx} \log |e^{ix}-1|dx  = -\frac{|k|}{\pi} \int_0^{2\pi} e^{ikx} \log (2\sin(x/2)) dx \nonumber \\
  &= -\frac{|k|}{\pi} \int_0^{2\pi} \cos (kx) \log (2\sin(x/2)) dx, 
  \end{align} 
  where the last step follows upon pairing $x$ and $2\pi -x$.  
  Since $\cos$ is an even function, it is enough to establish \eqref{lem2.2} in the case when $k$ is positive.  
 Integration by parts shows that the right-hand side of \eqref{lem2.2} equals 
  $$ 
  -\frac{1}{\pi} \int_0^{2\pi} \log (2\sin(x/2)) d\sin(kx) = \frac{1}{2\pi} \int_0^{2\pi} \frac{\sin kx}{\sin (x/2)} \cos (x/2) dx. 
  $$ 
  Since 
  $$ 
  \frac{\sin kx}{\sin (x/2)} = \frac{e^{ikx}-e^{-ikx}}{e^{ix/2}-e^{-ix/2}} = 2 \sum_{j=1}^{k} \cos \Big( \frac{2j-1}{2} x\Big),
  $$  
   it follows that 
   $$
   \frac{1}{2\pi} \int_0^{2\pi} \frac{\sin kx}{\sin (x/2)} \cos (x/2) dx 
   = \sum_{j=1}^{k} \frac{1}{\pi} \int_0^{2\pi} \cos \Big(\frac{2j-1}{2}x\Big) \cos (x/2) dx= 1,  
   $$ 
   which proves \eqref{lem2.2}, and therefore also \eqref{lem2.1} and \eqref{4.1}.  
   
   The triangle inequality gives 
   $$ 
   \Big| \sum_{j=1}^{N} e^{ik \theta_j} \Big| \le \frac{|k|}{\pi} \int_0^{2\pi} \Big| \log |P(e^{i\theta})| \Big| d\theta. 
   $$ 
   Now  
   \begin{align} 
   \label{4.3} 
   \frac{1}{2\pi} \int_{0}^{2\pi} \Big| \log |P(e^{i\theta})| \Big| d\theta &= 
   \frac{1}{2\pi}\int_0^{2\pi} \Big( 2\log^+ |P(e^{i\theta})| - \log |P(e^{i\theta})| \Big) d\theta \nonumber \\
   &= 2h(P),
   \end{align} 
   upon recalling the definition of $h(P)$, and upon noting that Jensen's formula gives $\int_0^{2\pi} \log |P(e^{i\theta})|d\theta =0$. 
   This establishes \eqref{4.2}.
%
\end{proof}

The reader familiar with Weyl's equidistribution theorem (see Chapter 3 of \cite{Korner} for an introduction)
 will recognize at once the significance of Lemma \ref{lem2}.  The estimate \eqref{4.2} shows that if $h(P)$ 
 is known to be small compared to $N$, then so are the power sums $\sum_{j=1}^{N} e^{ik\theta_j}$, at 
 least for small values of $k$.  Weyl's criterion then gives the equidistribution $\mod {2\pi}$ of the angles $\theta_j$.  
 Our goal now is to flesh out this argument; the general procedure is standard, but a few refinements are introduced 
 to obtain Theorem \ref{thm2} in its clean form.  
  
  Let $I$ be an arc on the unit circle, and let ${\mathcal I}(\theta)$ denote the indicator function for the arc $I$, 
  which is  $2\pi$-periodic.  Thus, ${\mathcal I}(\theta) = 1$ if $e^{i\theta} \in I$ and $0$ otherwise.   We are interested 
  in the number of zeros lying on the arc $I$: 
  $$ 
  N(I;P) =  \sum_{j=1}^{N} {\mathcal I}(\theta_j).
  $$ 
  Since ${\mathcal I}$ is periodic, it is tempting to invoke its Fourier expansion.  This is a little delicate, since the function ${\mathcal I}$ is 
  discontinuous and its Fourier series is not absolutely convergent.  Instead we will work with ``smoothed sums over zeros" $\sum_{j=1}^{N} g(\theta_j)$ where $g$ is a $2\pi$-periodic function with better behaved Fourier series, and then choose $g$ to be a suitable 
  approximation to the indicator function ${\mathcal I}$.  
  
   \begin{proposition} \label{lem4} Let $P(z) = \prod_{j=1}^{N} (z- e^{i\theta_j})$ be as above.  Let $g(\theta)$ be a $2\pi$-periodic continuous function such that 
  $$ 
  \sum_{k=-\infty}^{\infty} |k {\widehat g}(k) | < \infty, 
  $$ 
  where 
  $$ 
  {\widehat g}(k) = \frac{1}{2\pi } \int_0^{2\pi} g(\theta) e^{-ik\theta} d\theta
  $$ 
  denotes the Fourier coefficients of $g$.   Put 
  $$
  G(\theta) = \sum_{k=-\infty}^{\infty} |k| {\widehat g}(k) e^{ik\theta} \qquad \text{and} \qquad G= \max_{\theta} |G(\theta)|. 
  $$ 
  Then 
  $$ 
  \Big| \sum_{j=1}^{N} g(\theta_j) - \frac{N}{2\pi} \int_0^{2\pi} g(\theta) d\theta \Big| \le 4Gh(P). 
  $$ 
  \end{proposition} 
  
  If the $2\pi$-periodic function $g$ is $\ell$-times continuously differentiable, then integration by parts $\ell$ times 
  gives (for $k\neq 0$)
  $$ 
  |{\widehat g}(k) |= \Big| \frac{1}{(ik)^{\ell}} \frac{1}{2\pi } \int_{0}^{2\pi} g^{(\ell)}(\theta) e^{-ik\theta} d\theta\Big| 
  \le \frac{1}{|k|^{\ell}} \max_{\theta \in [0,2\pi)} |g^{(\ell)}(\theta)|.  
  $$ 
  Thus, for example, any thrice continuously differentiable function will meet the hypothesis of Proposition \ref{lem4} and there 
  is a rich supply of such functions.


%

  
    \begin{proof}[Proof of Proposition \ref{lem4}]    
  Using the Fourier expansion of $g$, we obtain 
  $$ 
   \sum_{j=1}^{N} g(\theta_j) - \frac{N}{2\pi} \int_0^{2\pi} g(\theta) d\theta = \sum_{k\neq 0} {\widehat g}(k) \sum_{j=1}^{N} e^{ik\theta_j}, 
   $$ 
  and so by Lemma \ref{lem2} this equals
  $$ 
  -\sum_{k\neq 0} {\widehat g}(k) \frac{|k|}{\pi} \int_0^{2\pi} e^{ik\theta} \log |P(e^{i\theta})| d\theta = 
  -\frac{1}{\pi} \int_0^{2\pi} \log |P(e^{i\theta})| \sum_{k\neq 0} |k| {\widehat g}(k) e^{ik\theta} d\theta. 
  $$
  From the definition of $G$, in magnitude the above is bounded by 
  \begin{align*}
  \frac{G}{\pi} \int_{0}^{2\pi} \Big| \log |P(e^{i\theta})| \Big| d\theta &= \frac{G}{\pi} \int_0^{2\pi} \Big( 2\log^{+}|P(e^{i\theta})| - \log |P(e^{i\theta})| \Big) d\theta \\
  &= 4G h(P),
  \end{align*} 
 upon recalling \eqref{4.3}. 
  \end{proof}

 To pave the way for the proof of Theorem \ref{thm2} in the next section, we work out the bound of Proposition \ref{lem4} for a particular class 
 of functions $g$.  The idea is that one can construct functions $g$ meeting the hypothesis of Proposition \ref{lem4} by convolving the indicator function ${\mathcal I}$ 
 with suitable nice functions ${\mathcal K}$.  In the next section, we shall make a specific choice for ${\mathcal K}$ so that the resulting function $g$ 
 approximates the indicator function ${\mathcal I}$ well.   
   
 \begin{lemma} \label{lem5}  Let $I$ be an arc on the unit circle,  and let ${\mathcal I}(\theta)$ denote its indicator function as above.  
 Let ${\mathcal K}$ be a $2\pi$-periodic continuous function that is always nonnegative, 
 and whose Fourier coefficients ${\widehat {\mathcal K}}(n)$ are all nonnegative, with $\sum_{n\in {\Bbb Z}} {\widehat {\mathcal K}}(n) < \infty$.  
 Let $g$ be the convolution of ${\mathcal I}$ and ${\mathcal K}$. Thus, 
  \begin{equation*} 
 \label{1.7}
  g(\theta) =\frac{1}{2\pi } \int_0^{2\pi} {\mathcal I}(\alpha) {\mathcal K}(\theta-\alpha) d\alpha. 
 \end{equation*} 
Then, $g$ satisfies the hypothesis of Proposition \ref{lem4}, and in the notation used there,  
$$ 
G = \max_{\theta}  \Big | \sum_{k=-\infty}^{\infty} |k| {\widehat g}(k) e^{ik \theta} \Big| \le \frac{2}{\pi^2} {\mathcal K}(0).
$$
 \end{lemma} 
 \begin{proof}     Suppose that ${\mathcal I}$ is the arc from $e^{i\alpha}$ to $e^{i\beta}$, so that for $k\neq 0$ we have 
 $$ 
 {\widehat {\mathcal I}}(k) = \frac{1}{2\pi } \int_{\alpha}^{\beta} e^{-ik y} dy = \frac{e^{-ik\alpha}- e^{-ik\beta}}{2\pi i k}.
 $$  
 The Fourier coefficients of the convolution of two functions are the products of the Fourier coefficients of those 
 functions; thus ${\widehat g}(k) = {\widehat {\mathcal I}}(k) {\widehat {\mathcal K}}(k)$.   Therefore $|k{\widehat g}(k)| =|k{\widehat {\mathcal I}}(k)| {\widehat {\mathcal K}}(k) \le {\widehat {\mathcal K}}(k)/\pi$ so that 
 $$ 
 \sum_{k=-\infty}^{\infty} |k {\widehat g}(k)| \le \frac{1}{\pi} \sum_{k=-\infty}^{\infty} {\widehat {\mathcal K}}(k) = \frac{1}{\pi} {\mathcal K}(0) < \infty. 
 $$ 
 This shows that the hypothesis $\sum_{k} |k{\widehat g}(k)| < \infty$ in Proposition \ref{lem4} is satisfied, and moreover establishes the bound $G \le {\mathcal K}(0)/\pi$.   
 
 To obtain the more precise bound for $G$ claimed in our lemma, note that 
 \begin{align*}
 G(\theta) &= \sum_{k=-\infty}^{\infty} |k| {\widehat g}(k) e^{i k \theta}  = 
 \sum_{k=-\infty}^{\infty} |k| {\widehat {\mathcal K}}(k) {\widehat {\mathcal I}}(k) e^{ik\theta} \\
& =\frac{1}{2\pi i} \sum_{\substack{k=-\infty \\ k\neq 0 }}^{\infty} \text{sgn}(k)  {\widehat {\mathcal K}}(k)  (e^{ik(\theta-\alpha)} - e^{ik(\theta-\beta)})
 .
 \end{align*}
 Pairing the terms $k$ and $-k$ together, and using ${\widehat {\mathcal K}}(k) = {\widehat {\mathcal K}}(-k)$ 
 (since ${\mathcal K}$ and ${\widehat {\mathcal K}}$ are real valued), we find  
 \begin{align*}
|G(\theta)|  
 &=\frac{1}{2\pi} \Big| \sum_{\substack{k=-\infty \\ k\neq 0 }}^{\infty} \text{sgn}(k) {\widehat {\mathcal K}}(k) \Big( \sin( k (\theta -\alpha) -\sin(k (\theta- \beta)\Big) \Big| 
\\
& \le 2 \max_{\phi} \frac{1}{2\pi} \sum_{k=-\infty}^{\infty} {\widehat {\mathcal K}}(k) |\sin (k\phi)| . 
 \end{align*}
  Therefore 
 \begin{equation*}
 \label{1.8}
 G= \max_{\theta} |G(\theta)| \le 2 \max_{\phi}\frac{1}{2\pi} \sum_{k=-\infty}^{\infty} {\widehat {\mathcal K}}(k) |\sin (k\phi)|. 
 \end{equation*} 
 A simple calculation gives the Fourier expansion 
 $$ 
|\sin x| = \frac{2}{\pi} - \frac{4}{\pi} \sum_{\ell =1}^{\infty} \frac{\cos(2\ell x)}{4\ell^2-1},
$$ 
and so  we obtain 
\begin{align*}
\frac{1}{2\pi} \sum_{k=-\infty}^{\infty} {\widehat {\mathcal K}}(k) |\sin (k\phi)| 
&= \frac{1}{2\pi} \sum_{k=-\infty}^{\infty} {\widehat {\mathcal K}}(k)  \Big(\frac{2}{\pi} - \frac{4}{\pi} \sum_{\ell =1}^{\infty} \frac{\cos(2k\ell \phi)}{4\ell^2-1}\Big) \\
&= \frac{1}{\pi^2} {\mathcal K}(0) - \frac{2}{\pi^2} \sum_{\ell=1}^{\infty} \frac{{\mathcal K}(2\ell\phi)}{4\ell^2-1} \le \frac{1}{\pi^2} {\mathcal K}(0), 
\end{align*}
proving the lemma.
 \end{proof}  
 
 \section{Proof of Theorem \ref{thm2}.} 
   Let $I$ be an arc on the unit circle.  To establish \eqref{1.4} it is enough to show that 
 \begin{equation} 
 \label{1.9} 
    N(I;P) - \frac{|I|}{2\pi} N \le \frac{8}{\pi} \sqrt{Nh(P)}. 
 \end{equation}  
 Once the upper bound is in place, we may use that 
 $$
 N(I;P)-|I|N/(2\pi) = |I^{c}|N/(2\pi) - N(I^c;P),
 $$
  where $I^c$ denotes the arc complementary to $I$, to obtain a corresponding lower bound, and thus complete the proof of 
 Theorem 2.

 Let $g$ be a $2\pi$-periodic function that majorizes the indicator function of $I$; that is, $g(\theta) \ge 0$ always, and $g(\theta) \ge 1$ if 
 $e^{i\theta} \in I$.   Then 
 \begin{align} 
 \label{1.10} 
  N(I;P) - &\frac{|I|}{2\pi} N \le \sum_{j=1}^N g(\theta_j) - \frac{|I|}{2\pi} N \nonumber \\ 
  &= 
\Big( \sum_{j=1}^{N} g(\theta_j) - \frac{N}{2\pi} \int_0^{2\pi} g(\theta)d\theta\Big) + N\Big( \frac{1}{2\pi} \int_{0}^{2\pi} g(\theta) d\theta -\frac{|I|}{2\pi} \Big).
 \end{align} 
Now the strategy is to find a nice function $g$ for which we can use Proposition \ref{lem4} and  Lemma \ref{lem5} to bound the first term above, 
 while also keeping $g$ close to the indicator function of $I$ so that the second term is also small.   
   
 Given $\pi >\delta >0$, let $I_{\delta}$ denote the arc obtained by widening $I$ on either side by $\delta$. (If $|I|+2\delta >2\pi$ then take $I_\delta$ to be all of the unit circle.)  Let us denote by ${\mathcal I}_\delta$ the indicator function 
  of the widened arc $I_{\delta}$. 
    Let ${\mathcal K}_\delta$ denote the $2\pi$-periodic function, given by 
  $$ 
  {\mathcal K}_\delta(\theta) = \frac{2\pi}{\delta^2} \max (\delta -|\theta|, 0) 
  $$ 
  for $\theta \in (-\pi, \pi]$.  
  The function ${\mathcal K}_\delta$ is closely related to the Fejer kernel (see, for example, Chapter 2 of \cite{Korner}), and its Fourier coefficients are easily computed: ${\widehat {\mathcal K}}_{\delta}(0) =  1$, and for $k\neq 0$ 
  $$ 
  {\widehat {\mathcal K}_{\delta}}(k) = \Big( \frac{\sin (k\delta/2)}{k\delta/2}\Big)^2. 
  $$ 
  Take $g$ to be the convolution of ${\mathcal I}_{\delta}$ and ${\mathcal K}_{\delta}$; thus 
   $
  g(\theta) = \frac{1}{2\pi } \int_0^{2\pi} {\mathcal I}_\delta(\alpha) {\mathcal K}_\delta(\theta-\alpha) d\alpha$. 
 From the definition of ${\mathcal K}_\delta$, and noting that $\frac{1}{2\pi }\int_0^{2\pi} {\mathcal K}_\delta(\alpha) d\alpha =1$, we see easily that the function $g(\theta)$ is always nonnegative, and it equals $1$ if $e^{i\theta} \in I$.   We may think of $g$ as the indicator function ${\mathcal I}$ ``smeared out" over a $\delta$ neighborhood of the arc $I$.  If we make $\delta$ smaller, our approximation $g$ 
 is closer to ${\mathcal I}$ and the second term on the right in \eqref{1.10} will become smaller, but, on the other hand, the function $g$ will become 
 ``less smooth" and the first term on the right in \eqref{1.10} will become larger.  The idea is to choose $\delta$ optimally so as to balance these two effects.     
 
Note that 
$$
{\widehat g}(0) = \frac{1}{2\pi } \int_0^{2\pi} g(\theta) d\theta = \frac{|I_\delta|}{2\pi} = \frac{|I|+2\delta}{2\pi}, 
$$ 
unless $I_\delta$ is all of the unit circle in which case ${\widehat g}(0)=1$.  Since $g$ majorizes the indicator function of $I$, we 
may use \eqref{1.10}, and from our evaluation of ${\widehat g}(0)$ it follows that the second term in the right side of \eqref{1.10} is at most $N\delta/\pi$.  

To bound the first term in \eqref{1.10}, we appeal to Proposition \ref{lem4} and Lemma \ref{lem5}.  They show that 
\begin{align*}
\Big| \sum_{j=1}^{N} g(\theta_j)  - \frac{N}{2\pi} \int_0^{2\pi} g(\theta) d\theta \Big| &\le 4 h(P) \max_{\theta} \Big| \sum_{k=-\infty}^{\infty} |k|{\widehat g}(k) e^{ik \theta } \Big| 
\\
&\le \frac{8}{\pi^2} {\mathcal K}_{\delta}(0) h(P) = \frac{16}{\pi\delta} h(P). 
\end{align*}   
We conclude that 
$$ 
\sum_{j=1}^{N} g(\theta_j) - \frac{|I|}{2\pi} N \le \frac{16}{\pi \delta} h(P) + \frac{\delta}{\pi}N, 
$$ 
and choosing $\delta = 4\sqrt{h(P)/N}$, the estimate \eqref{1.9} follows.  The proof of Theorem 2 is now complete.   

\smallskip 

We conclude by looking back at the proofs, and pointing out the key steps.  Theorem 1, showing that the roots accumulate near the unit circle, was a 
simple application of Jensen's formula.  The more difficult Theorem 2, which gives the equidistribution of the angles of the roots, began with an observation of Schur allowing us to restrict attention to the case when all roots lie on the unit circle.   Then the key identity is contained in Lemma 2, which connects power sums of the roots with the size of the polynomial on the unit circle.   Lemma 2 allows us to understand smooth sums over the angles of the roots, as in Proposition 1.   The last step is the passage from smooth sums over angles to identifying angles lying on arcs, and this is carried out in Lemma 3 together with the work of this section. 
 
   \medskip 
 \noindent {\bf Acknowledgment.} 
   I am grateful to Emanuel Carneiro, Persi Diaconis,  Andrew Granville, Emmanuel Kowalski,  Chen Lu, Pranav Nuti, and the referees for helpful comments, and 
especially to Pranav Nuti for producing Figure 1.  I am partially supported by a grant from the National Science Foundation, 
and a Simons Investigator grant from the Simons Foundation.  Part of the paper was written while the author was a Gauss Visiting Professor at 
G{\" o}ttingen; I thank the University, and the Akademie der Wissenschaften zu G{\" o}ttingen for their generous hospitality.


 \end{document}